\documentclass[10pt]{amsart}
\usepackage{geometry}                
\usepackage{graphicx}
\usepackage{enumerate}
\usepackage{tikz}

\usepackage{hyperref}

\usepackage{amssymb}
\usepackage[T1]{fontenc}
\usepackage[bitstream-charter]{mathdesign}

\newtheorem{thm}{Theorem}[section]
\newtheorem{lem}[thm]{Lemma}
\newtheorem{prop}[thm]{Proposition}
\newtheorem{defn}[thm]{Definition}

\newtheorem{exam}[thm]{Example}

\newtheorem{def-prop}[thm]{Definition-Proposition}

\newcommand{\FF}{\mathbb{F}} 

\newcommand{\FZ}{\mathbb{Z}}  
\newcommand{\FC}{\mathbb{C}}  
\newcommand{\FR}{\mathbb{R}}  
\newcommand{\FQ}{\mathbb{Q}}  
\newcommand{\FP}{\mathbb{P}}

\newcommand{\fa}{\mathfrak{a}}

\newcommand{\fb}{\mathfrak{b}}

\DeclareMathOperator{\Td}{Todd}
\DeclareMathOperator{\GL}{GL}
\DeclareMathOperator{\SL}{SL}

\DeclareMathOperator{\bbe}{\mathbf{e}}

\newcommand{\sm}{\begin{pmatrix}}
\newcommand{\esm}{\end{pmatrix}}

\newcommand{\td}{\Td} 
\title{Higher Hickerson formula}
\author{Jungyun Lee}
\author{Byungheup Jun}
\author{Hi-joon Chae}

\email{lee9311@ewha.ac.kr}
\address{Department of Mathematics, Ewha Womans University,
52 Ewhayeodae-gil, Seodaemun-gu, Seoul 120-750, Republic of Korea}
\thanks{J.L. was supported by the Basic Science Research Program through the National Research Foundation of Korea(NRF) funded by the Ministry of Education(NRF-2011-0023688) and (NRF-2009-0093827).}

\email{bhjun@unist.ac.kr} 

\address{Department of Mathematical Sciences, UNIST,  
UNIST-gil 50, Ulsan 689-798, Republic of Korea}
\thanks{B.J. was supported by NRF-2015R1D1A1A09059083.}

\email{hchae@hongik.ac.kr}
\address{Department of Mathematics Education, Hongik University, 
Seoul 121-791, Republic of Korea}
\thanks{H.C. was supported by 2014 Hongik University Research Fund.}

\date{June 7, 2016}
\begin{document}

\begin{abstract}
In \cite{Hi}, Hickerson made an explicit formula for Dedekind sums $s(p,q)$ in terms of the continued fraction of $p/q$.
We develop analogous formula for generalized Dedekind sums $s_{i,j}(p,q)$
defined in association with the $x^{i}y^{j}$-coefficient of the Todd power series of the lattice cone in $\FR^2$ generated by 
$(1,0)$ and $(p,q)$.
The formula generalizes Hickerson's original one and reduces to Hickerson's for $i=j=1$. 
In the formula, generalized Dedekind sums are divided into two parts: the integral $s^I_{ij}(p,q)$ and the fractional $s^R_{ij}(p,q)$.
We apply the formula to Siegel's formula for partial zeta values at a negative integer and obtain a new expression which 
involves only $s^I_{ij}(p,q)$  the integral part of generalized Dedekind sums.  This formula directly
generalize Meyer's formula for the special value at $0$. 
Using our formula, we present the table of the partial zeta value at $s=-1$ and $-2$ in more explicit form. 
Finally, we present another application on the equidistribution property of the fractional parts of the graph $\Big{(}\frac{p}{q},R_{i+j}q^{i+j-2} s_{ij}(p,q)\Big{)}$ for a certain integer $R_{i+j}$ depending on $i+j$.
\end{abstract}
\maketitle

\tableofcontents

\section{Introduction}
In \cite{Hi}, Hickerson obtained an explicit formula for Dedekind sum $s(p,q)$ in terms of the elements of the continued fraction of $p/q$, where the Dedekind sum $s(p,q)$ is defined as 
%
\begin{equation*}
s(p,q):=\sum_{k=0}^{q-1}\big(\big(\frac{k}{q}\big)\big)\big(\big(\frac{pk}{q}\big)\big).
\end{equation*}
Here $((-))$ denotes the sawtooth function(ie. for $x\in \FR$, $((x)) = x- [x] -\frac12$ if $x\not\in \FZ$ and $((x))=0$ otherwise).
Hickerson's formula is written as follows:
\begin{equation}\label{hik_formul}
12s(p,q)=\begin{cases}
  \frac{p-q_{n-1}}{q}+\sum_{i=1}^{n}(-1)^{i+1}a_i    & \text{if $n$ is even }, \\
     \frac{p+q_{n-1}}{q}+\sum_{i=1}^{n}(-1)^{i+1}a_i -3   & \text{if $n$ is odd}.
\end{cases}
\end{equation}
where $p_k/q_k= [a_0, \ldots, a_{k}]$ be the $k$-th convergent of $p/q$ and $a_i$ are the terms of the continued fraction of $p/q$:
\begin{equation}
\frac{p}{q} = [a_0,\ldots,a_n]  = a_0 + \cfrac{1}{a_1+\cfrac{1}{\ldots + \cfrac{1}{a_n}}}
\end{equation}
The idea of the proof is successive application of the celebrated reciprocity formula of Dedekind sums: 
$$
s(p,q) + s(q,p) = \frac{1}{12}(p/q + q/p + 1/pq) -1/4
$$
Hickerson used the formula to show the density of 
the graph of Dedekind sums $\{(p/q, s(p,q))\in \FR^2| (p,q)=1, q>0\}$ in \textit{loc.cit.}. It implies a conjecture of Rademacher saying that $s(p,q)$ is dense in $\FR$(\cite{RG}).

In Hickerson's formula \eqref{hik_formul}, the Dedekind sum is divided into  
 integral and  fractional parts. 
Both  are interesting in opposite directions. 
Integral part appears in
the formula for $\zeta$ or $L$-values of real quadratic fields.
Namely, Meyer's formula for partial zeta value of an ideal $\fa$ of a real quadratic field at $s=0$ is as follows
%
\begin{equation*}
\zeta(\fa,0)=\begin{cases}
 \frac{1}{12}\sum_{i=1}^{n}(-1)^{i+1}a_i    & \text{if $n$ is even }, \\
     \frac{1}{12}\sum_{i=1}^{n}(-1)^{i+1}a_i -3   & \text{if $n$ is odd}.
\end{cases}
\end{equation*}
where $a_i$ are terms of the continued fraction of a rational number $p/q$ in relation with  $\fa$ (cf. \cite{Me}, \cite{Van}, \cite{J-L2}, \cite{J-L4}).
We are not going to give exact description of the numbers appearing here as well as the definition of the partial zeta function but refer to Sec.\ref{partial_sec} of this article. 
The special value is  
exactly the integer part of $12s(p,q)$ as appeared in \eqref{hik_formul} but a universal denominator $12$ independent of $\fa$.
The occurrence of $12$ is explained in \cite{J-L3}.
The fractional part (mod 1) is given as a Laurent polynomial in $p$ divided by $q$ and is necessary in relating the Kloosterman sum to Dedekind sums:
\begin{equation*}
\sum_{\substack{0<p<q\\ gcd(p,q)=1}}\exp\left(2\pi i \left(12 s\left(p,q\right)\right)\right) = 
\sum_{p\in (\FZ/q\FZ)^*} \exp\left(2\pi i \frac{p+ p^{-1}}{q}\right)
\end{equation*}
Then the Weil bound for Kloosterman sum implies the Weyl equidistribution of the fractional part of $12 s(p,q)$ in the unit interval.

For partial zeta values at strictly negative integers, there is another formula due to Siegel(\cite{Siegel}, see also \cite{C-S}), but one needs generalized Dededkind sums:
\begin{equation}
s_{ij}(p,q):=\sum_{k=0}^{q-1}\bar{B}_{i}(\frac{k}{q})\bar{B}_{j}(\frac{pk}{q}).
\end{equation}
Here $\bar{B}_i(-)$ is the $i$-th periodic Bernoulli function. 
These sums $s_{ij}(p,q)$ vanish for odd $N=i+j$. 
For the Siegel formula, we refer the reader to Sec.\ref{partial_sec}. of this article.
The case one of $j=1$ is known as Apostol sums. These sums are studied by Apostol and Carlitz and appear in  modular transformation of  generalized Lambert series (cf. \cite{apostol}, \cite{apostol2}, \cite{ca1}, \cite{ca2}).


The goal of this paper is to obtain a higher degree analogue of the Hickerson's formula. 
It is an explicit formula for the generalized Dedekind sums using the 
terms of the continued fraction of $p/q$ as in Hickerson's. 
Furthermore, the formula should feature separation between the integral and the fractional part in the following way:
\begin{equation}
s_{i,j}(p,q)=\frac{1}{q^{i+j-2}}s_{i,j}^{I}(p,q)+\frac{1}{q^{i+j-1}}s_{i,j}^{R}(p,q).
\end{equation}
Here,  $s_{ij}^I$ and $s_{ij}^R$ denote the integral and the fractional part which are written in terms of the terms of the continued fraction of $p/q$ and  belong to $\frac{1}{R_{i+j}}\FZ$ where $R_{i+j}$ is an integer fixed by $N=i+j$ ($R_{i+j}$ is given explicitly in \eqref{Rij}).
Bearing in mind the two results aforementioned, which are obtained from the formulae of the integral and the fractional parts, we are going to present two applications in each direction. 

The main result of \cite{J-L} shows that 
$R_{i+j}s_{ij}^R(p,q)/q \pmod{1}$(thus same as $R_{i+1}q^{i+j-2} s_{ij}(p,q) \pmod{1}$) 
is equidistributed in the unit interval.
%
%
In \textit{loc. clit.},  explicit formula for $s_{i,j}^{R}(p,q)$ is given.
In this paper, it remains to find explicit formula of  $s_{i,j}^{I}(p,q)$, %
that will complement the previous result. 
For $i=j=1$, it is reduced to the case of the classical Dedekind sums as appeared in  Hickerson(\cite{Hi})
with $R_2=R_{i+j}=12$.
It is crucial ingredient to single out the integral part from the classical Dedekind sums in
the proof of Density of the graph of Dedekind sums in $\FR^2$. 
Analogous formula for the generalized Dedekind sums can be used in showing the equidistribution of the graph 
$\Big{(}\frac{p}{q},R_{i+j} q^{i+j-2} s_{ij}(p,q)\Big{)} \mod {\FZ\times \FZ}$. 
It will be an application of the main theorem and appear in Sec. \ref{last_sec}.
For classical Dedekind sums, one has much stronger  equidistribution for the fractional part by results of Myerson and Vardi(cf. \cite{Myerson}, \cite{Vardi}). One can take $R_{i+j}$ as arbitrary nonzero real.


%
When  $i,j\ne 1$, it is not so simple as in Hickerson's due to lack of nice reciprocity formula 
relating $s_{ij}(p,q)$ and $s_{ij}(q,p)$. 
If either $i$ or $j$ is $1$, the generalized Dedekind sums are studied by Apostol and Carlitz. 
There is a reciprocity formula found by  Apostol (cf. \cite{apostol}, \cite{apostol2}, \cite{ca1}, \cite{ca2}). 
In general as reciprocity, the best is the one involving other sums with same $N=i+j$ at the same time. 
We notice that the generalized Dedekind sums $s_{ij}(p,q)$ appear to be closely related to the coefficients of the Todd series  $\Td_{pq}(x_1,x_2)$ of a lattice cone generated by $e_1=(1,0)$ and $(p,q)$ (For the definition of Todd series, see Sec. \ref{sec_Todd}).
Additionally, the association of (normalized) Todd series $S_{\sigma}$ to a lattice cone $\sigma$ has additivity.
\begin{equation}
\begin{split}
\Phi:\sigma \mapsto S_\sigma(A_{\sigma}^{-1}(x,y)^{t}), \\
\Phi(\sigma(v_1,v_2)) + \Phi(\sigma(v_2,v_3)) = \Phi(\sigma(v_1,v_3)).
\end{split}
\end{equation}
Here, for two primitive lattice vector $v, w$, $\sigma(v,w)$ denotes the lattice cone generated by $v,w$.
In this context, the reciprocity formulae are all consequences of this additivity. 
The first quadrant cone of $\FR^2$ can be divided into two by putting a vector $(p,q)$ for $p,q$ relative prime and both positive. 
Then one can relate $\td_{pq}(x,y)$ and $\td_{qp}(x,y)$. 
Reading the coefficient of particular monomial, one can recover the reciprocity formula including other generalized Dedekind sums of the same weight(cf. \cite{Pom}, \cite{Solomon}). 

The strategy of this paper is to lift the reciprocity to the level of the Todd series.
As is well-known, the continued fraction of $p/q$ describes a decomposition of the cone given by $(p,q)$ into nonsingular cones. 
At the final step, we obtain the explicit formula for generalized Dedekind sums in terms of the continued fraction of $p/q$.


The first of our main result is the following explicit formula:
%

\begin{thm} \label{ma}
Suppose $e,f$ is a positive integer and $N=e+f$ is even.
Then we have 
\begin{equation*}
s_{e,f}(p,q)+\delta(e,f)B_{e} B_{f}=\frac{1}{q^{N-2}}s_{e,f}^I(p,q) + \frac{1}{q^{N-1}}s_{e,f}^R(p,q),
\end{equation*}
where $s_{e,f}^I(p,q)$ and $s_{e,f}^R(p,q)$ are given as:
\begin{equation*}
\begin{split} 
&s_{e,f}^I(p,q)=e!f! \sum_{\alpha+\beta=e} \Big[\sum_{k=-1}^{n-1}(-1)^{k+1}f_k(\alpha,\beta)+ 
\frac{B_N}{N!}\sum_{k=0}^{n-1}(-1)^ka_{k+1}g_k(\alpha,\beta)\Big{]}, \\
&s_{e,f}^R(p,q)= e!f!\frac{B_N}{N!}\left[\binom{N-1}{f}p^{f}+(-1)^{e(n-1)}\binom{N-1}{e}(q_{n-1})^{e}\right]
\end{split}
\end{equation*}
Here 
$$
D_k:=pq_k-qp_k,\quad  \text{for $k=0,1,\cdots, n-1$}
$$ and
\begin{equation*}
\begin{split}
&f_k(\alpha,\beta):=
\frac{q_k^{\beta}q_{k+1}^{\alpha}}{\alpha!\beta!}\sum_{i=\alpha}^{N-2-\beta} \frac{(-1)^i}{(N-2-i-\beta)!(i-\alpha)!}\frac{B_{i+1}}{(i+1)}\frac{B_{N-i-1}}{(N-i-1)}D_k^{N-2-i-\beta}D_{k+1}^{i-\alpha} \\
&g_k(\alpha,\beta):= \frac{q_{k-1}^{\beta}q_{k+1}^{\alpha}}{\alpha!\beta!}\sum_{i=\alpha}^{ N-2-\beta}\frac{(N-i-2)!}{(N-2-i-\beta)!}\frac{i!}{(i-\alpha)!}D_{k-1}^{N-i-2-\beta}D_{k+1}^{i-\alpha}
\end{split}
\end{equation*}
and
$\delta(i,j)=\begin{cases}1,& i=1 \text{ or }j=1,
\\0, &\text{ otherwise}\end{cases}$ and $B_i$ is the $i$-th Bernoulli number. 
\end{thm}

Using this formula, we present here two applications. 
First, we apply the explicit formula to the Siegel formula for the partial zeta value of an ideal of a real quadratic field at nonpositive integer and obtain a modified one.
Surprisingly, $s^R_{i,j}(p,q)$ in the formula cancels each other and 
the new formula does not involve $s^R_{i,j}(p,q)$ the fractional part of $s_{i,j}(p,q)$ at all.
As seen in Thm.\ref{ma}, it appears  as a polynomial in terms of the continued fraction of $p/q$ same as in Meyer's formula. 
Actually, restricting $i=j=1$, one can recover the formula of Meyer.

\begin{thm}[Higher Meyer formula]\label{higher_Meyer_formula}
Let $\fa = [\alpha,\beta]$ be an ideal of a real quadratic field and $\epsilon$ be the totally positive fundamental unit of $K$.
If the multiplication of $\epsilon^{-1}$ is identified with a matrix $\begin{pmatrix} p & q \\ r & s \end{pmatrix}\in \SL_2(\FZ)$ w.r.t. basis $\{ \alpha, \beta \}$,
then we have
\begin{equation}
\zeta(\fa,1-N)=\frac{(-1)^N}{q^{2N-2}}  \sum_{k=0}^{2N-2}(-1)^{k + 1} e_{N,  k} \frac{s^{I}_{k + 1,  2 N - k - 1}(p,q)}{(k + 1) (2 N - 1 - k)}
\end{equation}
where $e_{N,k}$ is given by $(z^2+(p+s)z+1)^{N-1}=\sum_{i=0}^{2N-2}e_{N,k}z^k$.
\end{thm}
%

The second application is the distribution of the graph of the fractional part $R_{i+j}q^{i+j-2} s_{i,j}(p,q) \pmod{1}$ as function of 
$p/q$. 
Using Thm.\ref{ma}., $R_{i+j}q^{i+j-2} s_{i,j}(p,q) \pmod{1}$ equals $R_{i+j} q^{i+j-2} s^R_{i,j}(p,q)$.
\begin{thm}[Equidistribution]\label{ma2}
The fractional parts of a graph $\Big{(}\frac{p}{q},R_{i+j} q^{i+j-2} s_{ij}(p,q)\Big{)}$ are equidistributed in $[0,1)^2$,
where
\begin{equation}\label{Rij}
R_{i+j}:=\begin{pmatrix}
        N \\
       i
\end{pmatrix}  \beta_{N} r_N,
\end{equation}
and $\frac{\alpha_k}{\beta_k}$ is the reduced fraction of $B_k\neq 0$ with $\beta_k>0$ 
and 
$$
r_N:=\text{L.C.M.}\left\{\text{Denominator of} \,\,\beta_{N}\begin{pmatrix}N \\i+1\end{pmatrix} B_{i+1}B_{N-i-1}\Big| \text{$i$ odd}, 0\leq i \leq {N-2}\right\}.
$$
\end{thm}

For higher dimension, we have an analogue of not only the Dedekind sums but also the generalized Dedekind sums (\cite{zagier1},
\cite{FY}, \cite{CJL}). 
In \cite{CJL}, the decomposition of higher dimensional Dedekind sums into integral and rational parts is given but only the rational part is explicit. 

This paper is composed as follows. 
In Sec. 2,  we briefly recall the notion of Todd series and its cocycle property. 
We relate the generalized Dedekind sums to the coefficients of Todd series.
In Sec. 3, we relate a cone decomposition to a continued fraction. 
In Sec.4, we use the results of earlier sections to obtain an explicit formula lifting the Hickerson's
formula to the level of Todd series.  In Sec. 5, we give the proof of Thm.\ref{ma} by descending
to the level of generalized Dedekind sums. 
Sec.6 is devoted to the proof of Thm.\ref{higher_Meyer_formula} and contains 
a table for partial zeta values in our formula. 
Finally in Sec.7, we prove Thm.\ref{ma2}.  

\section*{Acknowlegements}
We are grateful to  Vincent Maillot and Haesang Sun for various discussions and encouragement.

\section{Todd series and generalized Dedekind sums}\label{sec_Todd}
Let $M$ be the standard lattice $\FZ^2$ in $\FR^2$. 
Let $\sigma = \sigma(v_1,v_2)$ be a lattice cone with  primitive lattices $v_1,v_2$ bounding $\sigma$. 
We take the orientation(ie. the order of the rays) into consideration so that $$\sigma(v_1,v_2) =-\sigma(v_2,v_1).$$ 
Let $A_\sigma$  be an integer coefficient matrix whose columns are the lattice vectors $v_1,v_2$ in $\FZ^2$.
We will often identify with a lattice cone $\sigma$ with its matrix $A_\sigma$.
$M_\sigma$ denotes the sublattice of $M$ generated by $v_1,v_2$. 
$\Gamma_\sigma := M/M_\sigma$ is isomorphic to a cyclic group of order $\left|\det(A_\sigma)\right|$.

For $g \in M$ representing $\gamma\in \Gamma_\sigma$,  we have rational numbers $a_{\sigma,i}(g)$, $i=1,2$ such that
$$
g = a_{\sigma,1}(\gamma)v_1 + a_{\sigma,2}(\gamma) v_2 .
$$

Since $a_{i,\sigma}$ is  integral on $M_\sigma$, we have a character  $\chi_{\sigma,i}$ on $\Gamma_\sigma$
as  
$$
\chi_{\sigma,i} (\gamma) := \bbe\left(a_{\sigma, i}(g)\right),\quad \text{for $i=1,2$}.
$$

\begin{defn}\label{Todd_series}
The Todd power series of $\sigma$ is defined as:
\begin{equation}\label{todd_definition}
\Td_\sigma (x_1,x_2): = \sum_{\gamma\in \Gamma_\sigma} \frac{x_1}{1-\chi_{\sigma,1}(\gamma) e^{-x_1}}
\frac{x_2}{1-\chi_{\sigma,2}(\gamma) e^{-x_2}}.
\end{equation}
\end{defn}
The Todd series is analytic in a neighborhood of $0\in \FC^2$ and has Taylor expansion.
Note that the Todd series is invariant of  the $\GL_2(\FZ)$ equivalent class of  cones by the following proposition
(ie. for $A\in \GL_2(\FZ)$,  
$\Td_\sigma (x_1,x_2) = \Td_{A\sigma} (x_1,x_2)$).

If $p, q>0$ are a pair of nonnegative integers prime to each other,  $(1,0)$ and $(p,q)$ are
primitive lattice vectors and linearly independent. 
Let $\sigma_{pq}$ denote the cone generated by
$(1,0)$ and $(p,q)$.   
We shall write $\Td_{pq}$ instead of $\Td_{\sigma_{pq}}$ for simplicity.

Then the generalized Dedekind sum $s_{ij}(p,q)$ is identified with the $x_1^i x_2^j$-coefficient of
$\Td_{pq}$.
For $i=j=1$, it is the classical Dedekind sum: $s_{11}(p,q) = s(p,q)$.
Let us denote the $x_1^i x_2^j$-coefficient of $\Td_{pq}(x_1,x_2)$ by $t_{ij}(p,q)$.
\begin{equation}
\Td_{pq}(x_1,x_2) = \sum_{i,j \ge 0} \frac{t_{ij}(p,q)}{i! j!} x_1^i x_2^j.
\end{equation}
Then we have 
\begin{equation}
t_{i j }(p,q)=-(-q)^{i+j-1}\left(s_{ij}(p,q)+\delta(i,j)B_i B_j\right),
\end{equation}\label{equa10}
where 
$\delta(i,j)=\begin{cases}1,& i=1 \text{ or }j=1,
\\0, &\text{ otherwise}\end{cases}$ and $B_i$ is the $i$-th
Bernoulli number.


For a lattice cone $\sigma$, its normalized Todd series is defined as follows:
\begin{defn}
The  normalized Todd series $S_\sigma(x_1,x_2)$ of $\sigma$ is defined as
$$
S_{\sigma} (x_1,x_2) = \frac{1}{\det(A_\sigma) x_1 x_2 } \Td_\sigma(x_1,x_2).
$$
Similarly, $S_{\sigma_{pq}}$ is abbreviated to $S_{pq}$ as in unnormalized case.
\end{defn}

$S_\sigma(x_1,x_2)$ is a Laurent series in variables $x_1, x_2$ and has simple poles along $x_1=0$ and $x_2=0$.
Also note that swapping  two rays of the cone interchanges not only the variables but also the 
sign in $S_\sigma$. Thus the orientation  is reflected in $S_\sigma$ and changes the sign(cf. \cite{J-L3}).

\begin{align}
S_{-\sigma} (x_1,x_2) &= - S_{\sigma}(x_2,x_1) \\
\Td_{-\sigma} (x_1,x_2) & = \Td_\sigma(x_2,x_1)
\end{align}


\begin{prop}[Todd cocycle]
If we set 
$
\Phi(\sigma) = S_{\sigma} (A_\sigma^{-1}(x_1,x_2))$, then we have
$$
\Phi(\sigma_1 + \sigma_2) = \Phi(\sigma_1) + \Phi(\sigma_2).$$
\end{prop}

\begin{proof}
For the proof, we refer the reader to Thm.3 of \cite{Pom2}.
\end{proof}

\section{Continued fraction and Cone decomposition}
In this section, we would like  to relate nonsingular decomposition of $\sigma_{p,q}$ with the continued fraction of $p/q$
as is done in \cite{brion}. 
According to the cocycle property of the (normalized) Todd series, 
the (normalized) Todd series is decomposed into sum of those 
of nonsingular cones.

Suppose $q$ and $p$ be relatively prime positive integers  and suppose $q>p$.
Let $p/q$ have the following continued fraction expansion:
$$
\frac{q}{p}=a_1+ \cfrac{1}{a_2 + \cdots\cfrac{1}{a_n}},$$
where $a_i\geq1$ are all positive integers. 
Put $(p_{-1},q_{-1})=(1,0)$,
$(p_0,q_0)=(0,1)$. 
For $i\geq1$, $p_i/q_i$ be the $i$-th convergent of $p/q$:
$$
\frac{q_i}{p_i}:=a_1+ \cfrac{1}{a_2 + \cdots\cfrac{1}{a_i}}.
$$
$p_i$ and $q_i$ are uniquely determined by assuming $q_i>0$ and $\gcd(p_i,q_i)=1$.
$(p_i,q_i)$ is computed by iterating elementary M\"obius transformations $A_i =\left( \begin{smallmatrix}  a_i & 1 \\ 1 & 0\end{smallmatrix}\right)$ in $\GL_2(\FZ)$ in the following manner:
\begin{equation}
\begin{pmatrix} p_i \\ q_i  \end{pmatrix} = A_1 A_2 \cdots A_i \begin{pmatrix} 1 \\ 0\end{pmatrix} = A_1 A_2 \cdots A_{i-1} 
\begin{pmatrix}
a_i \\ 1
\end{pmatrix}
\end{equation}
Comparing $(p_k,q_k)$, $(p_{k-1},q_{k-1})$ and $(p_{k+1},q_{k+1})$ using above formula, we obtain 
\begin{equation}\label{relation_p_k}
(p_{k+1},q_{k+1}) = (p_{k-1},q_{k-1}) + a_{k+1} (p_k,q_k)
\end{equation}
which will be used in the next section.

Note that $v_k=(p_k,q_k)$ for $-1\leq k \leq n$ is  primitive lattice vector in the 1st quadrant.
We have  the following (virtual) cone decomposition of $\sigma$
into nonsingular cones: 
$$
\sigma:= \sigma_{pq}= \sigma(v_{-1},v_{n})=\sum_{k=-1}^{n-1} \sigma_k,
$$
where $\sigma_k := \sigma(v_k, v_{k+1})$.

\section{Nonsingular decomposition of Todd series}

From the additivity of normalized Todd series, we have 
$$\Phi(\sigma_{p,q})=\sum_{k=-1}^{n-2}\Phi(\sigma_k),$$
where $$\Phi(\sigma)=S_{\sigma}(A_{\sigma}^{-1}(x,y)).$$
For $\sigma_e=\sigma((1,0),(0,1))$, from  $\GL_2(\FZ)$-invariance of Todd power series,
we have 
\begin{equation*}
\begin{split}
&S_{\sigma_k}=S_{A_{\sigma_k }\sigma_e}=\frac{1}{\det{A_{\sigma_k}}xy} \Td_{A_{\sigma_k }\sigma_e}(x,y)\\
&=\frac{1}{\det{A_{\sigma_k}}xy} \Td_{\sigma_e}(x,y)=(-1)^{k+1}\frac{\Td(x,y)}{xy}.
\end{split}
\end{equation*}
Thus we obtain the following expression:
\begin{equation}\label{normalized_Todd}
S_{pq}(x,y)=\sum_{k=-1}^{n-1}(-1)^{k+1}F\left(A_{\sigma_k}^{-1}A_{\sigma}(x,y)^{t}\right),
\end{equation}
where
$$F(x,y)=\frac{1}{1-e^{-x}}\frac{1}{1-e^{-y}}=\frac{\Td (x,y)}{xy}.$$

For $-1\leq k \leq n-1$ let $M_k(x,y)$ be linear forms in $x$ and $y$ defined as follows:
\begin{equation}
A_{\sigma_k}^{-1} A_\sigma \begin{pmatrix} x \\ y \end{pmatrix}= \begin{pmatrix} M_{k+1}(x,y) \\ M_k(x,y)\end{pmatrix} 
\end{equation}
Since $A_{\sigma_k}^{-1}\in \GL_2(\FZ)$, $M_k(x,y)$ has coefficients in $\FZ$.

As $\det(A_\sigma)= q$, by multiplying $qxy$, we have 
\begin{equation}\label{S_sigma}
\begin{split}
\Td_{pq}(x,y)&=qxy\sum_{k=-1}^{n-1}(-1)^{k+1}\frac{\Td(M_k,M_{k+1})}{M_k M_{k+1}}\\
&=qxy\sum_{k=-1}^{n-1}(-1)^{k+1}\sum_{i=0}^{\infty}\sum_{j=0}^{\infty}(-1)^{i+j}\frac{B_i B_j}{i! j!}M_k^{j-1}M_{k+1}^{i-1}.
\end{split}
\end{equation}

Denote by $\Td_{\sigma}^N$ the degree $N$  homogeneous
part of $\Td_{\sigma}$. 
Then from \eqref{S_sigma} 
$\Td_\sigma^{N}$ is given as follows:
\begin{equation}\label{ToddN}
\begin{split}
\Td_{\sigma}^{N}
=&qxy\sum_{k=-1}^{n-1}(-1)^{k+1}\sum_{i=0}^{N-2}(-1)^N\frac{B_{i+1}}{(i+1)!}\frac{B_{N-i-1}}{(N-i-1)!}M_k^{N-2-i}M_{k+1}^{i}\\
&+(-1)^Nqxy\frac{B_{N}}{N!}\sum_{k=-1}^{n-1}(-1)^{k+1}\frac{M_k^{N}+M_{k+1}^{N}}{M_kM_{k+1}}.
\end{split}
\end{equation}

From \eqref{relation_p_k}, 
for $k\ge 0$ 
we have
$$
M_{k-1}-M_{k+1}=a_{k+1}M_k.
$$
and 
\begin{equation}\label{qxy}
\begin{split}
&qxy\sum_{k=-1}^{n-1}(-1)^{k+1}\frac{M_k^{N}+M_{k+1}^{N}}{M_kM_{k+1}}\\
&=qxy\left(\sum_{k=0}^{n-1}(-1)^ka_{k+1}\sum_{i=0}^{N-2}M_{k-1}^{N-2-i}M_{k+1}^i\right)+M_0^{N-1}x+M_{n-1}^{N-1}y.
\end{split}
\end{equation}

Therefore, plugging \eqref{qxy} into \eqref{ToddN},  we obtain
\begin{equation}\label{eq1}
\begin{split}
&\Td_{pq}^{N}
=qxy\sum_{k=-1}^{n-1}(-1)^{k+1}\sum_{i=0}^{N-2}(-1)^{N}\frac{B_{i+1}}{(i+1)!}\frac{B_{N-i-1}}{(N-i-1)!}M_k^{N-2-i}M_{k+1}^{i}\\
&+(-1)^Nq\frac{B_{N}}{N!}xy\left(\sum_{k=0}^{n-1}(-1)^ka_{k+1}\sum_{i=0}^{N-2}M_{k-1}^{N-i-2}M_{k+1}^i
\right)+\frac{B_{N}}{N!}M_0^{N-1}x+\frac{B_{N}}{N!}M_{n-1}^{N-1}y.
\end{split}
\end{equation}


\section{Computation of Dedekind sums}
In this section, using the result of earlier sections, we give a proof of Thm.\ref{ma}.

Let $m_k$ and $\ell_k$ be the coefficients of $M_k$:
$
M_k(x,y) = m_k x+\ell_k y.
$ 
We already know that $m_k$ and $\ell_k$ are determined by the continued fraction of $p/q$ in the following manner:
\begin{equation}
M_k(x,y) = m_k x+\ell_k y=
\begin{cases}
qy, &   k=-1\\
(-1)^{k}\left(q_kx+(p q_k-q p_k)y\right), & 0\le k \le n-1\\
(-1)^n qx, & k=n .
\end{cases}
\end{equation}
Then each part of \eqref{eq1} is computed as follows:
\begin{itemize}
\item
\begin{equation*}
\begin{split}
&xy\sum_{k=-1}^{n-1}(-1)^{k+1}\sum_{i=0}^{N-2}\frac{B_{i+1}B_{N-i-1}}{(i+1)!(N-i-1)!}(m_k x+\ell_ky)^{N-2-i}(m_{k+1} x+\ell_{k+1}y)^{i}\\
& = \sum_{k=-1}^{n-1}(-1)^{k+1}\sum_{\alpha,\beta}f_k(\alpha,\beta)x^{\alpha+\beta+1} y^{N-1-\alpha-\beta}
\end{split}
\end{equation*}
where 
$$
f_k(\alpha,\beta):=\frac{m_k^{\beta}m_{k+1}^{\alpha}}{\alpha! \beta!}\sum_{i=\alpha}^{N-2-\beta} \frac{1}{(N-2-i-\beta)!}\frac{1}{(i-\alpha)!}\frac{B_{i+1}}{(i+1)}\frac{B_{N-i-1}}{(N-i-1)}\ell_k^{N-2-i-\beta}\ell_{k+1}^{i-\alpha}  
$$
\item
\begin{equation*}
\begin{split}
&xy\left(\sum_{k=0}^{n-1}(-1)^ka_{k+1}\sum_{i=0}^{N-2}(m_{k-1}x+\ell_{k-1}y)^{N-i-2}(m_{k+1}x+\ell_{k+1}y)^i
\right)\\
&=\sum_{k=0}^{n-1}(-1)^ka_{k+1}\sum_{\alpha, \beta}g_k(\alpha,\beta)x^{\alpha+\beta+1} y^{N-1-\alpha-\beta},
\end{split}
\end{equation*}
where
$$
g_k(\alpha,\beta):= \frac{ m_{k-1}^{\beta}m_{k+1}^{\alpha}}{\alpha! \beta!}\sum_{i=\alpha}^{ N-2-\beta}\frac{(N-i-2)!}{(N-2-i-\beta)!}\frac{i!}{(i-\alpha)!}\ell_{k-1}^{N-i-2-\beta}\ell_{k+1}^{i-\alpha}.
$$
\item
$$
M_0^{N-1}x=\sum_{i=0}^{N-1}{N-1 \choose i} 
m_0^{i}\ell_0^{j}x^{i+1}y^{j}.
$$
\item
$$
M_{n-1}^{N-1}y=\sum_{i+j=N-1}{N-1 \choose i} 
m_{n-1}^i \ell_{n-1}^{j}x^{i}y^{j+1}
$$
\end{itemize}

\begin{thm}\label{tcoefficient}
The coefficient of $x^{e+1}y^{f+1}$ of  $\Td_{pq}^{N}$ with $e+f=N-2$ is
\begin{equation*}
\begin{split} 
& q \sum_{\alpha+\beta=e} \Big[\sum_{k=-1}^{n-1}(-1)^{k+1}f_k(\alpha,\beta)+ \frac{B_N}{N!}\sum_{k=0}^{n-1}(-1)^ka_{k+1}g_k(\alpha,\beta)\Big{]} \\
& + \frac{B_N}{N!}\binom{N-1}{f+1}p^{f+1}+\frac{B_N}{N!} \binom{N-1}{e+1}((-1)^{n-1}q_{n-1})^{e+1}.
\end{split}
\end{equation*}
\end{thm}
Combining the formula (10) and Thm. \ref{tcoefficient}, we have the proof of Thm.\ref{ma}.

\section{Application 1: partial zeta values}\label{partial_sec}
In this section, we apply the formula for generalized Dedekind sums to the formula for partial zeta values of ideals of real quadratic fields at nonpositive integers due to Siegel 
and finish the proof of Thm.\ref{higher_Meyer_formula}.

Let us recall the definition of the partial zeta function of an ideal of a number field. 
For an ideal $\fa$ of a number field $K$, the partial zeta function $\zeta(\fa, s)$ is defined as
$$
\zeta(\fa,s) := \sum_{\substack{\fb \sim \fa \\ \fb : \text{integral ideal}}} N\fb^{-s}, \quad\text{for $Re(s)>1$}.
$$
$\zeta(\fa,s)$ is meromorphically continued to the whole complex plane with a unique simple pole at $s=1$. 

Let  $\fa=[\alpha,\beta]$ be an ideal of a real quadratic field $K$. 
Let $\omega=-\frac{\beta}{\alpha}>\omega'$ and $\epsilon>1$ 
be the totally positive fundamental unit. 
With respect to the basis $\alpha, \beta$, the multiplication by $\epsilon^{-1}$ is represented by a hyperbolic element 
$\begin{pmatrix}
p & r \\ q & s
\end{pmatrix}\in \SL_2(\FZ)$:
$$ \epsilon^{-1} \begin{pmatrix}
      \alpha   \\
        \beta
\end{pmatrix}=\begin{pmatrix}
    p&q   \\
     r&s
\end{pmatrix}\begin{pmatrix}
      \alpha   \\
        \beta
\end{pmatrix}.$$

In this setting, a reformulation of the  formula of Siegel by Coates-Sinnott(Lemma 11  in \cite{C-S}) is given as follows:
\begin{thm}[Siegel's formula]
\begin{equation}
\begin{split}
(-1)^N\frac{2N}{B_{2N}}&\zeta(\fa,1-N)=\\
&\frac{1}{q^{2N-1}} \sum_{k=0}^{2N-2}(-1)^{k } \frac{e_{N,  k}}{k + 1} (p + s)^{k + 1}+ \frac{B_{2N}}{2N}\sum_{k=0}^{2N-2}(-1)^{k + 1} e_{N,  k} \frac{s_{k + 1,  2 N - k - 1}(p,q)}{(k + 1) (2 N - 1 - k)}
\end{split}
\end{equation}
where $e_{N,k}$ is given in the following way:
$$
(z^2+(p+s)z+1)^{N-1}=\sum_{i=0}^{2N-2}e_{N,k}z^k
$$
\end{thm}

Now we apply our formula in Thm.\ref{ma} to Siegel's formula.
Soon we will see   that the contribution of the rational part $s_{e,f}^R(p,q)$ of Dedekind sums  canceled by themselves. 
It is essentially due to the following observation. 
\begin{lem}\label{rational}
$$  \sum_{k=0}^{2N-2}(-1)^{k } \frac{e_{N,  k}}{k + 1} (p + s)^{k + 1}+\frac{2N}{B_{2N}}\sum_{k=0}^{2N-2}(-1)^{k + 1} \frac{ e_{N,  k}}{(k + 1) (2 N - 1 - k)}s^{R}_{k + 1,  2 N - k - 1}(p,q)=0$$
\end{lem}
\begin{proof}
We note that from Theorem 1.1, we have 
$$
s^{R}_{e, f}(p,q)= \frac{B_{e + f}}{(e + f)}  \Big{(}e p^
    f + f  \Big{(}(-1)^{n-1}q_{n-1}\Big{)}^e\Big{)}.$$
 Also we have   $\quad s=(-1)^{n-1}q_{n-1}$ (See pp. 39--40, \cite{Van}).
 
Putting these into the left hand side of the equality, we have
 \begin{equation*}
 \begin{split}
  T:=& \sum_{k=0}^{2N-2}(-1)^{k } \frac{e_{N,  k}}{k + 1} (p + s)^{k + 1}+\frac{2N}{B_{2N}}\sum_{k=0}^{2N-2}(-1)^{k + 1} \frac{ e_{N,  k}}{(k + 1) (2 N - 1 - k)}s^{R}_{k + 1,  2 N - k - 1}(p,q)\\
  &= \sum_{k=0}^{2N-2}(-1)^{k } \frac{e_{N,  k}}{k + 1} (p + s)^{k + 1}+ \sum_{k=0}^{2N-2}(-1)^{k + 1} \frac{ e_{N,  k}}{(k + 1) (2 N - 1 - k)}((k+1)p^{2N-k-1}+(2N-k-1)s^{k+1}).
  \end{split}
  \end{equation*}
  
We notice that
$$
\int_{0}^{Z}(z^2+(p+s)z+1)^{N-1}dz=\sum_{i=0}^{2N-2}\frac{e_{N,k}}{k+1}Z^{k+1}
$$
and 
$$e_{N,2N-2-j}=e_{j},\,\,\,\,\,\, j=0,1,\cdots,2N-2.$$

Now $T$ can be written as sum of definite integrals as follows: 
 \begin{equation*}
 \begin{split}
  T&=- \sum_{k=0}^{2N-2}(-1)^{k + 1} \frac{e_{N,  k}}{k + 1} (p + s)^{k + 1}+ \sum_{k=0}^{2N-2}(-1)^{k + 1} \frac{ e_{N,  k}}{(k + 1) (2 N - 1 - k)}((k+1)p^{2N-k-1}+(2N-k-1)s^{k+1})\\
  &=-\int_{0}^{-p-s}(z^2+(p+s)z+1)^{N-1}dz+\int_{0}^{-p}(z^2+(p+s)z+1)^{N-1}dz+\int_{0}^{-s}(z^2+(p+s)z+1)^{N-1}dz\\
  &=-\int_{-p}^{-p-s}(z^2+(p+s)z+1)^{N-1}dz+\int_{0}^{-s}(z^2+(p+s)z+1)^{N-1}dz\\
  &=-\int_{0}^{-s}(z^2+(-p+s)z+1-ps)^N-(z^2+(p+s)z+1)^N dz.
  \end{split}
  \end{equation*}

Finally it remains  to show that 
$$
\int_{0}^{-s}(z^2+(-p+s)z+1-ps)^N-(z^2+(p+s)z+1)^N dz =0.
$$
By change of  the variable $z=t-s/2$, the above becomes
$$
\int_{s/2}^{-s/2} \left( (t^2 - p t + A)^N - (t^2 + pt + A)^N \right)dt
$$
where $A=1-ps/2 - s^2/4$. 
Further the integrand is simplified as 
$$
-2pt \sum_{i+j=N-1} (t^2 -pt +A)^i (t^2+pt +A)^j.
$$
It is easily checked to be an odd function in $t$. 
Thus the vanishing of $T$ is obtained and this completes the proof. 
\end{proof}
Now simply applying the above Lemma to Siegel's formula, we 
finish the proof of Thm.\ref{higher_Meyer_formula} in the following:
%
%
\begin{equation*}
\begin{split}
&\frac{(-1)^N2N}{B_{2N}}\zeta(\fa,1-N)\\
&=\frac{-1}{q^{2N-1}} \sum_{k=0}^{2N-2}(-1)^{k + 1} \frac{e_{N,  k}}{k + 1} (p + s)^{k + 1}\\
&\quad+\frac{1}{q^{2N-1}}\frac{2N}{B_{2N}} \sum_{k=0}^{2N-2}(-1)^{k + 1}  \frac{e_{N,  k}}{(k + 1) (2 N - 1 - k)}\left(q s^I_{k + 1,  2 N - k - 1}(p,q)+ s^R_{k + 1,  2 N - k - 1}(p,q)\right)\\
&=\frac{-1}{q^{2N-1}} \sum_{k=0}^{2N-2}(-1)^{k + 1} \frac{e_{N,  k}}{k + 1} (p + s)^{k + 1}+\frac{1}{q^{2N-1}}\frac{2N}{B_{2N}} \sum_{k=0}^{2N-2}(-1)^{k + 1}  \frac{e_{N,  k}}{(k + 1) (2 N - 1 - k)} s^R_{k + 1,  2 N - k - 1}(p,q)\\
&\quad+\frac{(-1)^N}{q^{2N-2}} \frac{2N}{B_{2N}} \sum_{k=0}^{2N-2}(-1)^{k + 1} e_{N,  k} 
\frac{s^{I}_{k + 1,  2 N - k - 1}(p,q)}{(k + 1) (2 N - 1 - k)}\\
&=\frac{(-1)^N}{q^{2N-2}} \frac{2N}{B_{2N}} \sum_{k=0}^{2N-2}(-1)^{k + 1} e_{N,  k} \frac{s^{I}_{k + 1,  2 N - k - 1}(p,q)}{(k + 1) (2 N - 1 - k)}
\end{split}
\end{equation*}

Next comes explicit formula of the partial zeta values for $s=-1$ and $s=-2$  using integral part of Dedekind sum and obtain explicit expression for them using continued fraction of $\frac{p}{q}.$
\begin{exam}
\begin{equation*}
\begin{split}
&\zeta(\fa,-1)=-\frac{1}{3q^3} s_{1,3}^I(p,q)+\frac{(p+s)}{4q^3}s_{2,2}^I(p,q)-\frac{1}{3q^3}s_{3,1}^I(p,q)\\
&\zeta(\fa,-2)\\
&=\frac{1}{5q^5}s_{1,5}^I(p,q)-\frac{1}{4q^5}(p+s)s_{2,4}^I(p,q)+\frac{1}{9q^5}(p+s)^2s_{3,3}^I(p,q)-\frac{1}{4q^5}(p+s)s_{4,2}^I(p,q)+\frac{1}{5q^5}s_{5,1}^I(p,q).
\end{split}
\end{equation*}

Below is  the table of explicit formula of $s_{i,j}^I(p,q)$  for $i+j=4$ and $ 6$, respectively.
\begin{center}
    \begin{tabular}{| l |  l  |}
      \hline
 $s_{2,2}^I(p,q)$&$\frac{1}{36}\sum_{k=-1}^{n-1}(-1)^{k+1} \Big{(}D_{k+1}q_{k}+D_k q_{k+1}\Big{)}$\\
 $$&$-\frac{1}{180}\sum_{k=0}^{n-1}(-1)^ka_{k+1}\Big{(}q_{k+1}(D_{k-1}+2D_{k+1})+q_{k-1}(2D_{k-1}+D_{k+1})\Big{)}$\\
  \hline
 $ s_{3,1}^I(p,q)$&$\frac{1}{24}\sum_{k=-1}^{n-1}(-1)^{k} q_kq_{k+1}$\\
 $$&$-\frac{1}{120}\sum_{k=0}^{n-1}(-1)^ka_{k+1}\Big{(}q_{k+1}^2+q_{k-1}q_{k+1}+q_{k-1}^2\Big{)}$\\
  \hline
      $ s_{1,3}^I(p,q)$&$\frac{1}{24}\sum_{k=-1}^{n-1}(-1)^{k} D_kD_{k+1}$\\
      $$&$-\frac{1}{120}\sum_{k=0}^{n-1}(-1)^ka_{k+1}\Big{(}D_{k+1}^2+D_{k-1}D_{k+1}+D_{k-1}^2\Big{)}$\\
\hline
\end{tabular}
\end{center}

\begin{center}
    \begin{tabular}{| l |  l  |}
      \hline
$ s_{5,1}^I(p,q)$&$\frac{1}{72}\sum_{k=-1}^{n-1}(-1)^{k+1} \Big{(}q_{k+1}q_{k}^3+q_k q_{k+1}^3\Big{)}$\\
$$&$+\frac{1}{252}\sum_{k=0}^{n-1}(-1)^{k}a_{k+1}\Big{(}q_{k-1}^4+q_{k-1}^3q_{k+1}+q_{k-1}^2q_{k+1}^2+q_{k-1}q_{k+1}^3+q_{k+1}^4\Big{)}$\\
\hline
$ s_{4,2}^I(p,q)$&$\frac{1}{180}\sum_{k=-1}^{n-1}(-1)^{k+1} \Big{(}D_{1 + k} q_k^3 + 3 D_k q_k^2 q_{1 + k} + 
 3 D_{1 + k} q_{k} q_{1 + k}^2 + D_{k} q_{1 + k}^3\Big{)}$\\
 $$&$+\frac{1}{630}\sum_{k=0}^{n-1}(-1)^{k}a_{k+1}\Big{(}4 D_{k-1} q_{ k-1}^3 + D_ {k+1} q_{k-1}^3 + 
 3 D_{ k-1} q_{k-1}^2 q_{k+1} + 
 2 D_{k+1} q_{k-1}^2q_{k+1}$\\
 $$&$+2 D_{k-1} q_{k-1} q_{k+1}^2 + 
 3 D_{k+1} q_{k-1} q_{k+1}^2 + D_{k-1} q_{k+1}^3 + 
 4 D_{k+1} q_{k+1}^3\Big{)}$\\
\hline
$ s_{3,3}^I(p,q)$&$\frac{1}{80}\sum_{k=-1}^{n-1}(-1)^{k+1} \Big{(}D_kD_{1+k}q_k^2+D_k^2q_kq_{k+1}+D_{k+1}^2q_kq_{1+k}+D_kD_{1+k}q_{1+k}^2\Big{)}+$\\
$$&$\frac{1}{840}\sum_{k=0}^{n-1}(-1)^{k}a_{k+1}\Big{(}6D_{k-1}^2q_{k-1}^2+3D_{k-1}D_{k+1}q_{k-1}^2+D_{k+1}^2q_{k-1}^2+3D_{k-1}^2q_{k-1}q_{1+k}$\\
$$&$+4D_{k-1}D_{k+1}q_{k-1}q_{k+1}+3D_{k+1}^2q_{k-1}q_{k+1}+D_{k-1}^2q_{k+1}^2+3D_{k-1}D_{k+1}q_{k+1}^2+6D_{1+k}^2q_{1+k}^2
\Big{)}$\\
\hline
$s_{2,4}^I(p,q)$&$\frac{1}{180}\sum_{k=-1}^{n-1}(-1)^{k+1} \Big{(}3D_k^2D_{k+1}q_k+D_{k+1}^3q_k+D_k^3q_{1+k}+3D_kD_{1+k}^2q_{1+k}\Big{)}$\\
$$&$+\frac{1}{630}\sum_{k=0}^{n-1}(-1)^ka_{k+1}\Big{(}4D_{k-1}^3q_{k-1}+3D_{k-1}^2D_{k+1}q_{k-1}+2D_{k-1}D_{k+1}^2q_{k-1}+D_{k+1}^3q_{k-1}$\\
$$&$+D_{k-1}^3q_{k+1}+2D_{k-1}^2D_{k+1}q_{k+1}+3D_{k-1}D_{k+1}^2q_{k+1}+4D_{k+1}^3q_{k+1}\Big{)}$\\
\hline
$s_{1,5}^I(p,q)$&$\frac{q}{72}\sum_{k=-1}^{n-1}(-1)^{k+1} \Big{(}D_{k+1}D_{k}^3+D_k D_{k+1}^3\Big{)}$\\
$$&$+\frac{q}{252}\sum_{k=0}^{n-1}(-1)^ka_{k+1}\Big{(}D_{k-1}^4+D_{k-1}^3D_{k+1}+D_{k-1}^2D_{k+1}^2+D_{k-1}D_{k+1}^3+D_{k+1}^4\Big{)}$\\
\hline
\end{tabular}
\end{center}

\end{exam}

\section{Application 2: equidistribution of the graph of Generalized Dedekind sum}\label{last_sec}

In this section, we give a proof for Thm.\ref{ma2} which describes  the distribution of $H_{ij}(p,q)$ for $i+j$ even and a relatively prime pair of integers $(p,q)$ in $\FR^2$ given as follows:
$$
H_{ij}(p,q):=\Big{(}\frac{p}{q},R_{i+j} q^{N-2}s_{i,j}(p,q)\Big{)}
$$ 
We  will show that the fractional part of the graph $H_{ij}(p,q)$ taken in $[0,1)^2$ are 
equidistributed in the sense of Weyl(\cite{Weyl}).
As a function on the unit torus, it is equivalent to saying that the average of the point-mass supported on the sequence of the vectors $H_{ij}(p,q)$(ordered by increment of $q$) is weakly 
convergent to the standard Lebesgue measure on the torus.

A   sequence $\alpha_k$ in $[0,1)$ is said to be equidistributed if the average of the point-mass supported on the
truncated sequence weakly converges to the standard Lebesgue measure of the unit interval. 
It is checked by the behavior of the  Fourier coefficient of the average of the point-mass. 
Let $\bbe(x)=\exp(2\pi i x)$ for $x\in \FR$. 
The Fourier coefficient of the character $\bbe(mx)$ is given as
$$
\int_0^1 \bbe(m x) \frac1N \sum_{k\le N} \delta_{\alpha_k} dx = \frac{1}{N}\sum_{k\le N} \bbe(m \alpha_k).
$$
Then the sequence $\alpha_k$ is equidistributed iff $a_m(N) \to 0$ as $N\to \infty$ for each nonzero $m\in \FZ$.

For a sequence on a $n$-dimensional torus to be equidistributed, 
the criterion is changed as follows. 
We note first that a character of $n$-dimensional  is represented by a lattice vector $\bf{m}\in \FZ^n$.  
It needs that for any $\bf{m}\in \FZ^n\setminus \{0\}$,
$$
\frac1N \sum_{k\le N} \bbe(\bf{m}\cdot \alpha_k) \to 0 \quad\text{as $N\to \infty$}.
$$



Now we return to the case of our interest.
Let $\bf{m}\in \FZ^{2}$ a lattice vector and let, for a positive real number $x$, 
$E({\bf m},x)$ be the average of the exponential of $(2\pi i) {\bf m}\cdot  H_{ij}(p,q)$:
\begin{equation}
E({\bf m},x) := \frac{1}{\# \left\{(p,q)|\gcd(p,q)=1, p<q\le x \right\} } \sum_{0<q<x}\sum_{\substack{0<p<q\\(p,q)=1}} \bbe\left( {\bf m}\cdot H_{ij}\left(p,q\right)\right).
\end{equation}
For $H_{ij}(p,q) $ to be equidistributed in $T^2$, it needs to check that
$E({\bf m},x)\to 0$ as $x\to \infty$ for every nonzero $\bf{m}$.
This will be shown by relating  an exponential sum. 

Now, we relate 
\begin{equation}
\sum_{\substack{0<p<q\\(p,q)=1}} \bbe\left( {\bf m} \cdot H_{ij}\left(p,q\right)\right)
\end{equation}
to an exponential sum for a Laurent polynomial. 

 For a Laurent polynomial $F(x)\in \FZ[x,x^{-1}]$, we denote its exponential sum over $\FZ/q\FZ$ by
 $K(F,q)$:
\begin{equation}\label{G_ES}
K(F,q):=
\sum_{x\in (\FZ/q\FZ)^*}\bbe_q\left(F(x)\right),
\end{equation}
where $\bbe_q(x):=\bbe(\frac{x}{q})= \exp(2\pi i \frac{x}q)$.


For $H_{ij}(p,q)$, we recall the formula for the fractional part of $s_{ij}(p,q)$.
\begin{thm}[Jun-Lee\cite{J-L3}]\label{general}
For positive integers $i,j$ with $i+j=N\geq 2$ even, we have
$$
R_N q^{N-2}s_{ij}(p,q)  - \alpha_N r_N \frac{p'^{i}\binom{N-1}{i} +p^{j}\binom{N-1}{j}}{q} \in \FZ,
$$
where $p'$ is an integer such that $p'p \equiv 1 \pmod{q}$.
\end{thm}

Thus one can relate the exponential sum below to ${\bf m}\cdot H_{ij}(p,q)$.
$$
\sum_{\substack{0<p<q\\(p,q)=1}} \bbe\left( {\bf m}\cdot  H_{ij}\left(p,q\right)\right)=K({\bf m}\cdot F_{ij},q).
$$
where  $F_{ij}(x)$ is the rank-$2$ vector of  Laurent polynomials
$$
F_{ij}(x)=\left(x, x^{-i}\alpha_Nr_N\binom{N-1}{i}+x^{j}\alpha_Nr_N\binom{N-1}{j}\right),
$$
and  for ${\bf m}=(m_1,m_2)\in\FZ^2\setminus\{ 0\}$,
\begin{equation}\label{Fij}
{\bf m}\cdot F_{ij}(x)= \sum_{i=1}^{N-1}m_1x+m_2 x^{-i}\alpha_Nr_N\binom{N-1}{i}+m_2 x^{j}\alpha_Nr_N\binom{N-1}{j}.
\end{equation}

To estimate the exponential sum $K({\bf m}\cdot F_{ij},q)$,
we recall a result from SGA 4${}^1{\mskip -5mu/\mskip -3mu}_2$(\cite{SGA4.5}).
Consider an exponential sum of a Laurent polynomial $f(x)$ over $\FP^1(\FF_p)$ for a prime $p$:
$$
S'_f:= \sum_{\substack{x\in \FP^1(\FF_p)\\ f(x) \ne\infty}} \bbe_p(f(x))
$$
The irregularity index $v^*_x(f)$  at $x$ is actually the Swan conductor of the associated $\ell$-adic sheaf to $f$.
It is a local invariant defined in the following way:
\begin{enumerate}[(a)]
\item
$ 
v_x(f) =  \begin{cases} \text{the order of pole at $x$  of $f$ if $f(x)=\infty$} \\
0,\quad\text{otherwise}
\end{cases} 
$ 
\item
$v_x^*(f) : = \inf_g  v_x(f+ g^p-g)$
\end{enumerate}
A modified exponential sum is defined from $S'_f$ as
$$
S_f := \sum_{\substack{x\in\FP^1(\FF_p) \\ v^*_x(f) = 0}} \bbe_p\left(f\left(x\right)\right)
$$
where 
the value of $\bbe_p\left(f(x)\right)$ is extended to $\{x|v^*_x(f)=0\}$ by putting
$$
\bbe_p\left(f\left(x\right)\right) := \bbe_p\left( \left(f + g^p -g\right)\left(x\right) \right)
$$ 
for $g$ which makes 
$f + g^p -g$ regular at $x$.

%
%
Applying the formula (3.5.2) of \textit{loc.cit.}, we obtain a Weil type bound of $S_f$:
\begin{equation}
|S_f| \le -2 + \sum_{v^*_x(f)\ne 0} (1 + v_x^*(f)) p^{1/2}
\end{equation}


In our case, $f(x)= {\bf m}\cdot F_{ij}(x)$ explicitly given in  \eqref{Fij}. 
$$
v^*_{0}(f) = 
\begin{cases}
0 & \text{if $m_2 =0$}  \\
i &  \text{Otherwise}
\end{cases}
\quad\text{and}\quad
v^*_{\infty}(f) = 
\begin{cases}
j & \text{if $j>1$ and $m_2\ne 0$} \\
1 & \text{if $m_2 =0$}\\
0 & \text{if $j=1$ and $m_1 + m_2 = 0$}
\end{cases}
$$
%
For any nonzero ${\bf m}$, 
we have an upper bound for the exponential sum:
\begin{equation}
\left| K({\bf m}\cdot F_{ij}, p)\right| \le C \cdot p^{1/2}
\end{equation}
for a given prime $p$ and $C$ is a constant independent of $p$.

Now, we consider the case that $q$ is a power of prime $p$ or product of several prime powers. 
Using the same reduction process  in Section 6 of \cite{J-L3}, we  obtain the following bound for 
given positive integer $q$.

\begin{thm}
For given positive integer $i,j$ such that  $N=i+j$ is even and positive integer $q$, we have
$$
\left|K({\bf m}\cdot F_{ij}, q)\right|<< q^{\frac{3}{4}+\epsilon}
$$
for all $\epsilon>0$.
\end{thm}


Finally the Weyl criterion for $H_{ij}(p,q)$ comes from the following estimate
\begin{equation}
\begin{split}
\sum_{0<q<x}\sum_{\substack{0<p<q\\(p,q)=1}} 
\bbe\left({\bf m}\cdot  H_{\bf ij}(p,q) \right)= \sum_{0<q<x} K({\bf m}\cdot F_{ij},q) 
\leq x^{\frac{7}{4}+\epsilon}. 
\end{split}
\end{equation}

For $x>1$ let $\phi(x):=\left|\left(\FZ/[x]\FZ\right)^*\right|$ be the Euler's phi function. 
Since 
$
\sum_{q<x} \phi(q)\sim x^2,
$
the criterion is fulfilled for
the fractional part of the vector 
$H_N(p,q)$
\begin{equation*}
\begin{split}
&E(m,x)=  \frac{1}{\# \left\{(p,q)|(p,q)=1, 0< p<q\le x \right\} } \sum_{0<q<x}\sum_{\substack{0<p<q\\(p,q)=1}} \bbe
(m\cdot R_N H_{ij}(p,q))
\rightarrow0,
\end{split}
\end{equation*}
as $x\rightarrow \infty$.
Therefore  the proof is done.

\end{document}